\documentclass[english]{amsart}

\usepackage{amsmath,amssymb,enumerate}
\usepackage{amscd,amsxtra,calc}

\usepackage[T1]{fontenc}
\usepackage[all]{xy}

\usepackage{babel}
\usepackage{amstext}
\usepackage{amsmath}
\usepackage{amsfonts}
\usepackage{latexsym}
\usepackage{ifthen}
\usepackage{verbatim}
\usepackage{enumerate}

\usepackage{url}

\usepackage{xypic}
\xyoption{all}
\newcommand{\id}{{\rm id}}

\renewcommand{\O}{{\mathcal O}}

\newtheorem{lemma1}{}[section]

\newenvironment{lemma}{\begin{lemma1}{\bf Lemma.}}{\end{lemma1}}
\newenvironment{example}{\begin{lemma1}{\bf Example.}\rm}{\end{lemma1}}

\newenvironment{theorem}{\begin{lemma1}{\bf Theorem.}}{\end{lemma1}}
\newenvironment{proposition}{\begin{lemma1}{\bf Proposition.}}{\end{lemma1}}
\newenvironment{corollary}{\begin{lemma1}{\bf Corollary.}}{\end{lemma1}}
\newenvironment{remark}{\begin{lemma1}{\bf Remark.}\rm}{\end{lemma1}}

\newenvironment{remark*}{{\bf Remark.}}{}
\newenvironment{example*}{{\bf Example.}}{}

\newcommand{\R}{\ensuremath{\mathbb{R}}}
\newcommand{\Q}{\ensuremath{\mathbb{Q}}}

\newcommand{\N}{\ensuremath{\mathbb{N}}}
\newcommand{\PP}{\ensuremath{\mathbb{P}}}

\newcommand{\holom}[3]{\ensuremath{#1:#2  \rightarrow #3}}
\newcommand{\fibre}[2]{\ensuremath{#1^{-1} (#2)}}

\makeatletter
\ifnum\@ptsize=0  \fi \ifnum\@ptsize=2  \fi 

\newcommand\sO{{\mathcal O}}

\usepackage{xcolor}

\newcommand{\pic}[0]{\operatorname{Pic}}

\newcommand{\Hom}[0]{\operatorname{Hom}}

\newcommand{\NE}[1]{ \ensuremath{ \overline { \mbox{NE} }(#1)} }

\title{Normal split divisors in rational homogeneous spaces}
\date{\today}


\author {Enrica Floris}
\author {Andreas H\"oring}

\address{Enrica Floris, Universit\'e de Poitiers, Laboratoire de Math\'ematiques et Applications,  UMR~CNRS 7348, Boulevard Marie et Pierre Curie, Site du Futuroscope, TSA 61125, 86073
Poitiers Cedex 9, France and Institut Universitaire de France}
\email{enrica.floris@univ-poitiers.fr}

\address{Andreas H\"oring, Universit\'e C\^ote d'Azur, CNRS, LJAD, France}
\email{Andreas.Hoering@univ-cotedazur.fr}

\begin{document}

\begin{abstract}
We show that a divisor in a rational homogenous variety 
with split normal sequence is the preimage of a hyperplane section in either the projective space or a quadric.
\end{abstract}

\maketitle

%\tableofcontents

\section{Introduction}

Let $X$ %\simeq G/P$ 
be a rational homogeneous space, 
and let $\holom{i}{M}{X}$ be a submanifold of dimension at least two\footnote{See Remark \ref{remark-curve-case} for a discussion of the case of curves.}. We say that
the $M$ is normal split in $X$ if the normal sequence 
$$
0 \rightarrow T_M \stackrel{{\rm d} i}{\rightarrow} T_X \otimes \sO_M \rightarrow N_{M/X} 
\rightarrow 0
$$ 
admits a splitting morphism $s: T_X \otimes \sO_M \rightarrow T_M$ such that $s \circ {\rm d}i = \id_{T_M}$. It is easy to see that $T_M$ is then globally generated and therefore $M$ is itself a homogeneous space. In our paper \cite{FH24} we used the contact structure
of $\PP(T_X)$ to show that $M$ is actually {\em rational homogeneous}, i.e. $M$ has no abelian factor. If one specifies the structure of $X$, much more can be shown:
\begin{itemize}
\item If $X \simeq \PP^n$, then $M$ is a linear subspace \cite{vdv59};
\item If $X \simeq Q^n \subset \PP^{n+1}$ is a quadric, then $M$ is a linear space or a linear section of $X$ \cite{Jah05}.
\item If $X$ is an irreducible Hermitian symmetric space, then so is $M$ \cite{Din22}.
\end{itemize}  

In view of these surprisingly short lists it is tempting to believe that normal split submanifolds (and their embeddings) can be classified up to projective equivalence.
In this paper we we build on an ingenious argument of Beauville and M\'erindol \cite{BeaMer87} to obtain a vast generalisation of the divisorial case in van de Ven's
and Jahnke's theorem:

\begin{theorem}\label{thm:NSdivisor} Let $X$ be a rational homogeneous space, and let
$M \subset X$ be a smooth prime divisor that is normal split.
Then there exists a locally trivial fibration
$\varphi: X \rightarrow U$ such that $U \simeq \PP^n$ or $U \simeq Q^n$ 
and $M$ is the pullback of a smooth divisor $N \subset U$ that is normal split.\\
Moreover, if $n\geq 3$ then $N$ is a hyperplane section and if $n=2$ then $N$ is either a hyperplane section or $Y=\PP^1\times\PP^1$ and there is a non-negative integer $a$ such that $N$ has bidegree $(1,a)$ or $(a,1)$.
\end{theorem}

The most difficult case in the proof of this statement is when the divisor $M$ is ample. 
In this case we have $X=U$ and our goal is to show that $X$ is the projective space or a smooth quadric. The key observation is that the divisor $M \subset X$ satisfies
the technical conditions in the paper \cite{BeaMer87}; and therefore it is the fixed locus 
of an involution $\sigma$ on $X$. In particular
we have a two-to-one cover $X \rightarrow X/\sigma$. If
the Picard number of $X$ is one we conclude using a result of Hwang and Mok 
\cite{HM99}. If the Picard number is at least two we study the Mori contractions on $X$ to obtain a contradiction (unless $\dim X=2$ and $X$ is a quadric).

For the general situation note that 
a prime divisor $M$ in a rational homogeneous space $X$ always defines a basepoint free linear system and therefore induces a fibration 
$$
\varphi:= \varphi_{|M|}: X \rightarrow U
$$
onto a variety of dimension $n := \kappa(X, M)$.  Let $N \subset U$ be the smooth ample divisor such that  $M=\varphi^{-1}(N)$. 
We show that $N \subset U$ normal split and conclude with a case distinction
based on our analysis of the ample case.

{\bf Acknowledgements.} 
The authors are partially supported by the project ANR-23-CE40-0026 ``Positivity on K-trivial varieties''.
AH was supported by the France 2030 investment plan managed by the National Research Agency (ANR), as part of the Initiative of Excellence of Universit\'e C\^ote d'Azur under reference number ANR-15-IDEX-01.
EF was supported by the ANR Project FRACASSO ANR-22-CE40-0009-01 and the Institut Universitaire de France.

\section{Notation and basic facts}

We work over the complex numbers, for general definitions we refer to \cite{Har77}. 
Varieties will always be supposed to be irreducible and reduced. 
We use the terminology of \cite{Deb01, KM98}  for birational geometry and notions from the minimal model program. We follow \cite{Laz04a} for algebraic notions of positivity.

If $X$ is a projective variety, then $\mbox{Aut}^{\circ}(X)$ denotes the connected component of $\mbox{Aut}(X)$ containing the identity.

A fibration is a surjective projective morphism $f: X \rightarrow Y$ with connected fibres between normal varieties such that $\dim X>\dim Y$.

For a submanifold $i\colon M \hookrightarrow X$ of  a complex manifold $X$, we will simply denote the tangent map by
$$
\holom{{\rm d}i}{T_M}{T_X \otimes \sO_M}.
$$
We recall the three basic examples of normal split submanifolds:

\begin{example} \label{example-degree-two}
Let $\holom{\tau}{X}{Y}$ be a degree two branched covering between complex manifolds.
Then the ramification divisor $R \subset X$ is normal split. In fact the tangent map
${\rm d}\tau: T_X \rightarrow \tau^* T_Y$ has rank $\dim X-1$ along $R$ and its image
is $\tau^* T_B$ where $B \subset Y$ is the branch divisor.
Since $\tau|_R$ is an isomorphism onto $B$, the map
$$
T_X \otimes \sO_{R} \rightarrow \tau^* T_B \simeq T_R
$$
defines a splitting.
\end{example}

\begin{example} \label{example-fibration}
Let $\holom{\tau}{X}{Y}$ be a fibration between complex manifolds, and let
$\sigma: Y \rightarrow X$ be a section. Then $M:= \sigma(Y) \subset X$
is normal split. Indeed since $\tau \circ \sigma=\id_Y$ it is clear
that the restriction of the tangent map $T_\varphi$ to $M$ defines a splitting.

In particular if $X=\PP^1 \times \PP^1$ and $M \in |\sO_{\PP^1 \times \PP^1}(1,a)|$
is a smooth curve, then $M \subset X$ is normal split.
\end{example}

\begin{example} \label{example-fibres}
Let $\holom{f}{X}{Y}$ be a fibration between rational homogeneous manifolds, and let $F$ be a fibre of $f$.
Then $F$ is normal split. Indeed, since $X$ is homogeneous $f$ is smooth and its fibres are all biholomorphic. By the Fischer-Grauert theorem $f$ is locally trivial and thus $F$ is normal split.
In particular $F$ is rational homogeneous, for example by \cite{FH24}.
\end{example}

\section{Normal split divisors}

The goal of this section is to show the following key step towards Theorem \ref{thm:NSdivisor}:

\begin{theorem}\label{thm:NSampledivisor}
 Let $X$ be a rational homogeneous space of dimension $n\geq 3$, and let
$M \subset X$ be a smooth ample divisor that is normal split.
Then either $X \simeq \PP^n$ or $X \simeq Q^n$. In both cases $M$ is a hyperplane section.
\end{theorem}

\subsection{Involutions and the case of Picard rank 1}

We reformulate results in \cite{BeaMer87} in the following proposition.

\begin{proposition}\label{pro:BeaMer}
 Let $X$ be a rational homogeneous space of dimension $n\geq 3$ and let
$M \subset X$ be a smooth ample divisor that is normal split.
Assume that $(X, \sO_X(M))$ is not $(\PP^n, \sO_{\PP^n}(1))$.
Then there is an involution of $X$ with fixed locus $M$.
\end{proposition}

\begin{proof}
By \cite[Theorem 3.11]{Ram87} the ring of global sections $\bigoplus_{k \in \N} H^0(X, \sO_X(kM))$ is generated in degree one, so $M$ is very ample and defines an embedding
$$
i : X \hookrightarrow \PP^N
$$
such that $M = X \cap H$ with $H \subset \PP^N$ a hyperplane.
By \cite[Theorem 3.11(i), Remark 3.12]{Ram87} the variety $X \subset \PP^N$ is an intersection of quadrics, so we can apply \cite[Corollaire 2, Remarque (3)]{BeaMer87}. For the convenience of the reader let us sketch their remarkable argument: 
note first that the splitting map $s_M: T_X \otimes \sO_M \rightarrow T_M$ 
induces a map $N_{M/X} \rightarrow T_X \otimes \sO_M$.
Since $M \in |\sO_X(1)|$ the splitting of the normal sequence is thus given by a non-zero morphism $\holom{j}{\sO_M(1)}{T_X \otimes \sO_M}$.
Consider the exact sequence
$$
0 \rightarrow T_X(-1) \otimes \sO_M \rightarrow T_{\PP^N}(-1) \otimes \sO_M
\stackrel{u}{\rightarrow} N_{X/\PP^N}(-1) \otimes \sO_M \rightarrow 0
$$
The section $j$ induces a section $s \in H^0(M, T_{\PP^N}(-1) \otimes \sO_M)$
such that $u(s)=0$. Now consider the twisted Euler sequence restricted to $M$:
$$
0 \rightarrow \sO_M(-1) \rightarrow \sO_M^{\oplus N+1} \rightarrow T_{\PP^N}(-1) \otimes \sO_M \rightarrow 0.
$$
Since $\dim M \geq 2$ we have by Kodaira's theorem $H^0(M, \sO_M(-1))=H^1(M, \sO_M(-1))=0$. Thus the restriction map
$$
H^0(M, \sO_M^{\oplus N+1}) \rightarrow H^0(M, T_{\PP^N}(-1) \otimes \sO_M)
$$
is an isomorphism. Repeating the argument for the twisted Euler sequence we obtain an isomorphism
$$
H^0(\PP^N, T_{\PP^N}(-1)) \cong  H^0(M, T_{\PP^N}(-1) \otimes \sO_M).
$$
Thus up to linear coordinate change the section $s$ is the restriction of
an Euler vector field $l(x) \frac{\partial}{\partial X_0}$ with $l$ a linear form.

{\em Step 1. We claim that if $x \in M$, then the projective tangent space $\PP(T_{X,x}) 
\subset \PP^N$ passes through the point $p:= (1:0: \ldots:0)$.}

Let $V \subset \PP^N$ be the vanishing set of the Euler vector field, i.e.
the union of the hyperplane $l=0$ and a point.
Note that $M \not\subset V$ since otherwise $s=0$. In order to show the claim we can assume without loss of generality that $x \in (M \setminus V)$, the general case will just follows by passing to the limit.

We can choose homogeneous polynomials $F_1, \ldots, F_k$ vanishing on $X$ such that
$\PP(T_{X,x}) = \cap_{j=1}^k \PP(T_{V(F_j), x})$ where 
$\PP(T_{V(F_j), x}) \subset \PP^N$ is the hyperplane defined by the equation
$\sum_{i=0}^N \frac{\partial F_j}{\partial X_i}(x) X_i=0$.
It is clearly sufficient
to show that 
$$
p \in \PP(T_{V(F_j), x}).
$$
Yet this is equivalent to showing that $\frac{\partial F_j}{\partial X_0}(x)=0$.
Since $x \not\in V$, the vector field  $l(x) \frac{\partial}{\partial X_0}$
does not vanish in a neighbourhood of $x$. Yet the condition $u(s)=0$ is equivalent
$s \in H^0(M, T_X(-1) \otimes \sO_M)$, so the
derivation $\frac{\partial F_j}{\partial X_0}(x)$ vanishes.

{\em Step 2. We have $p \not\in M$.}
We argue by contradiction. 
By Step 1 we have $p \in \PP(T_{X,x})$ for every $x \in M$.
Since $M = X \cap H$ is smooth we have $\PP(T_{M,x}) = \PP(T_{X,x}) \cap H$
for every $x \in M$. Thus we have
$p \in \PP(T_{M,x})$ for every $x \in M$.
Take now $H_1, \ldots, H_{\dim M-1}$ general hyperplane sections passing through
$p$. Then $C:= M \cap H_1 \cap \ldots H_{\dim M-1}$
is a smooth curve such that $p \in \PP(T_{C,x})$ for every $x \in C$.
By \cite[IV, Prop.3.9]{Har77} the curve $C$ is a line and therefore $(X, \sO_X(M)) \simeq (\PP^n, \sO_{\PP^n}(1))$, a contradiction to our assumption.

{\em Step 3. Conclusion.}
Let $\holom{\sigma}{\PP^N}{\PP^N}$ be the involution fixing pointwise the
point $p$ and the hyperplane $H$. Let $Q \subset \PP^N$
be a quadric that contains $X$, then $\sigma(Q)=Q$ by \cite[Cor.1]{BeaMer87}.
Since $X$ is an intersection of quadrics we obtain
that $\sigma(X)=X$ and therefore the sought involution is $\sigma_X:=\sigma\vert_X$. 
We are left to show that the fixed locus of $\sigma_X$ is $M$, i.e.
we have to show that $p \not\in X$.
Yet again by \cite[Cor.1]{BeaMer87} a quadric $Q$ with $X \subset Q$
and $p \in Q$ is singular in $p$. Thus $X$ would be an intersection of quadric cones
with vertex the point $p$, so itself a cone. Yet $X$ is smooth,
so we get again $(X, \sO_X(M)) \simeq (\PP^n, \sO_{\PP^n}(1))$, the final contradiction.

\end{proof}

\bigskip

It is well-known that if $B \subset \PP^N$ is a smooth hypersurface, the variety $B$
is rational homogeneous if and only if $\deg B \leq 2$. We will now combine this basis fact with a difficult result of Hwang and Mok:

\begin{corollary}\label{cor:Q2proj}
Let $X$ be a rational homogeneous space of dimension $n$ and Picard number 1.
Let $\tau\colon X\to Y$ be a degree two finite morphism ramified along a smooth divisor $H$ of $X$. Then $X$ is a smooth quadric
\end{corollary}

\begin{proof}
By \cite[Main Thm.]{HM99} the manifold $Y$ is isomorphic to $\mathbb P^n$.
The ramification divisor is a rational homogeneous space because it is normal split in $X$. Thus the branch locus $B \simeq R \subseteq \mathbb P^n$ is homogeneous
and therefore of degree at most two.
Moreover $B$ is divisible by 2 in $\pic(\mathbb P^n)$, so $B$ is a smooth quadric in $\mathbb P^n$.
The covering $\tau$ is thus induced by the data $\O_{\PP^n}(H)^{\otimes 2}\simeq \O_{\PP^n}(B)$,
where $H$ is the hyperplane class. 
The ramification formula
$K_X=\tau^*(-(n+1)H+H)$ implies that $-K_X$ is divisible by  $n$.
By the Kobayashi-Ochiai criterion this shows that $X$ is a quadric.
\end{proof}

\begin{corollary}\label{cor:Pic1}
Let $X$ be a rational homogeneous space of dimension $n \geq 3$  of Picard number 1, and let
$M \subset X$ be a smooth ample divisor that is normal split.
Then either $X \simeq \PP^n$ or $X \simeq Q^n$ and $M$ is a hyperplane section.
\end{corollary}

\begin{proof}
If $X \simeq \PP^n$ this is van den Ven's theorem, so assume 
$X \not\simeq \PP^n$.
By Proposition \ref{pro:BeaMer} there is an involution $\sigma\colon X\to X$
whose fixed locus is $M$.
By Corollary \ref{cor:Q2proj} the variety $X$ is a smooth quadric and $M$ a hyperplane section.
\end{proof}

\subsection{Case of higher Picard rank}

The goal of this subsection is to show that in the setup of Theorem \ref{thm:NSampledivisor} the Picard number of $X$ is one. If $\rho(X)>1$ the Proposition \ref{pro:BeaMer} 
still provides an involution $\sigma: X \rightarrow X$, but the theorem of Hwang and Mok \cite{HM99} does not apply. Thus we study the Mori fibrations on $X$
and their interaction with the involution $\sigma$ to obtain a contradiction. 
Note that Chi-Hin Lau \cite{Lau09} generalised \cite{HM99} to rational homogeneous spaces $X=G/P$ with $G$ a {\em simple} Lie group, thereby providing an alternative approach to this setting.

\begin{proposition} \label{proposition-contractions}
Let $X$ be a rational homogeneous space of Picard number at least two, 
and let $M \subset X$ be a smooth ample divisor that is normal split.
\begin{itemize}
\item If $\dim X=2$ (and hence $X \simeq \PP^1 \times \PP^1$), then $M$ is a section
of one of the rulings.
\item If $\dim X \geq 3$, there exists an elementary Mori contraction $X \rightarrow Y$ induced by an extremal ray $\Gamma$ such that
$(K_X+M) \cdot \Gamma<0$.
Moreover, for each contraction of this type we have $\dim X \geq 2 \dim Y$.
\end{itemize}
\end{proposition}

\begin{proof}
Without any assumption on the dimension, observe first that $K_M= (K_X+M)|_M$ is not nef
since the normal split divisor $M \subset X$ is rational homogeneous \cite{FH24}.
Applying the cone theorem to the Fano manifold $X$ we obtain
that there exists an extremal ray $\Gamma$ such that $(K_X+M) \cdot \Gamma<0$, denote the corresponding contraction by $f: X \rightarrow Y$.

{\em 1st case. Assume that $\dim X-\dim Y=1$.} Then the general fibre $F$ is $\PP^1$,
so $K_X \cdot F=-2$ and $M \cdot F>0$ implies that $M \cdot F=1$.
Thus $M$ is a rational section of $f$. Since $M$ is homogeneous the birational
morphism $M \rightarrow Y$ is an isomorphism, hence $M$ is even a section.
By \cite[Appendix, Lemma (1)]{Bor91} this implies that $\dim X \leq 2$.
Thus we have $\dim X=2$ and $M$ is a section of the ruling $f$.

{\em 2nd case. Assume that $\dim X-\dim Y>1$.}
Note that this implies $\dim X>2$ since otherwise we would have $\rho(X)=1$.
Let $F$ be a general fibre of $F$. 
Since $(F \cap M) \subset F$ is an ample divisor of dimension at least one, it is connected.
Thus $f$ induces a fibration $\holom{f|_M}{M}{Z}$. Since $M$ is rational homogeneous the fibration $f|_M$ has maximal rank. Thus a theorem of Sommese \cite[Prop.V]{Som76} implies
that $\dim X \geq 2 \dim Y$.
\end{proof}

\begin{lemma}\label{lem:monodromy}
Let $X$ be a rational homogeneous space and let $f\colon X\to Z$ be a fibration.
Then for a fibre $F$ one has $\rho(F)=\rho(X)-\rho(Z)$.
\end{lemma}

\begin{proof}
Since $X$ is rational homogeneous, the fibration $f$ is smooth.
Thus by \cite[Thm.2.2]{CFST16}
the local system $R^1 f_*\mathbb G_m \otimes \mathbb Q$
is a local system on $Z$ with finite monodromy.
Since $Z$ is simply connected, it is a trivial local system.
The statement follows.
\end{proof}

Let us recall that if $X$ is $\mathbb Q$-factorial and $X \rightarrow Y$ is a finite surjective morphism
between normal varieties, then $Y$ is $\mathbb Q$-factorial if this holds for $X$ \cite[Lemma 5.16]{KM98}.

\begin{lemma}\label{lem:sigmastar}
Let $X$ be a $\mathbb Q$-factorial normal projective variety.
 Let $\sigma\colon X\to X$ be a finite order automorphism and let $\tau\colon X\to Y:=X/\langle\sigma\rangle$ be the induced Galois cover.

 \begin{enumerate}
  \item If $\sigma_*\colon N_1(X)_{\mathbb R}\to N_1(X)_{\mathbb R}$ leaves  each extremal ray of $NE(X)$ invariant,
  then $\sigma_*$ is the identity.
   \item If $\sigma_*\colon N_1(X)_{\mathbb R}\to N_1(X)_{\mathbb R}$ is the identity map then $\tau_*\colon N_1(X)_{\mathbb R}\to N_1(Y)_{\mathbb R}$ induces an isomorphism $NE(X)\to NE(Y)$.
 \end{enumerate}
\end{lemma}

\begin{proof}
Fix a $k \in \N$ such that $\sigma^k = \id$.

 We start with (1).
 Let $R$ be an extremal ray, then by assumption $\sigma_*(R)=R$.
 Choose $0 \neq v\in R$, then $v$ is an eigenvector with eigenvalue $\lambda \in \R$. 
 Since $\sigma^k=id$, we have $\sigma_*v=\pm v $.
 Yet $\sigma_*R\subseteq NE(X)$, so we have $\sigma_*v=v $.
 The extremal rays of $NE(X)$ span $N_1(X)_{\mathbb R}$, therefore $\sigma_*$ is the identity.

 As for (2), we prove first that $\tau^*\colon N^1(Y)_{\mathbb R}\to N^1(X)_{\mathbb R}$
 is an isomorphism. The map $\tau^*$ is always injective. To show that it is surjective,
let $[D]\in N^1(X)_{\mathbb R}$. Since $\sigma_*$ and therefore $\sigma^*$ is the identity we have
$$
[kD]=\left(\sum_0^{k-1} (\sigma^i)^*[D]\right)=\tau^*\tau_*[D],
$$
proving that $D=\tau^*\tau_*[1/k D]$.

Moreover, $\tau^*$ induces an isomorphism between the nef cones of $X$ and $Y$.
Passing to the dual cones, we get the claim.
\end{proof}

\begin{remark} \label{remark-picard}
Let $\tau: X \rightarrow Y$ be a double cover between projective manifolds of dimension $n \geq 4$. Then $\tau$ induces an isomorphism
from the ramification divisor $R \subset X$ onto the branch locus $B \subset Y$. 
Assume that $R$ (or equivalently $B$) is an ample divisor. 
Since $n>3$ we know by the Lefschetz hyperplane theorem that
$$
\rho(X)=\rho(R)=\rho(B)=\rho(Y).
$$
\end{remark}

We will frequently use the following technical statement:

\begin{lemma} \label{lemma-embedding}
Let $X$ be a rational homogeneous space, and let 
$$
\holom{f_i}{X}{Z_i}
$$
be distinct contractions of extremal rays on $X$. Then the map 
$$
\holom{f:=f_1 \times f_2}{X}{Z_1 \times Z_2}
$$
is $\mbox{Aut}^{\circ}(X)$-equivariant and an embedding, i.e. an isomorphism onto its image.
\end{lemma}

\begin{proof}
The fibrations $f_i$ are $\mbox{Aut}^{\circ}(X)$-equivariant by Blanchard's lemma, so it is clear
that $f$ is also equivariant. 
 For every point $(z_1, z_2) \in f(X)$, we have
$\fibre{f}{(z_1, z_2)}= \fibre{f_1}{z_1} \cap \fibre{f_2}{z_2}$. Since the extremal rays are distinct, their fibres intersect in at most finitely many points.
This implies that $f\colon X \to f(X)$ is finite.
Since $\mbox{Aut}^{\circ}(X)$ acts transitively on $X$ the equivariance of $f$ implies that
$f(X)$ is smooth and $f: X \rightarrow f(X)$ is \'etale. Since $f(X)$ is rationally connected and smooth, it is simply connected. Thus the \'etale map
$f: X \rightarrow f(X)$ is an isomorphism.
\end{proof}

\begin{lemma}\label{lem:Pic2}
Let $X$ be a rational homogeneous space, and
let $M \subset X$ be a smooth ample divisor.
Assume that there is a finite morphism of degree two $\tau \colon X\to Y$ ramified along $M$ such that $\sigma_*$ is not the identity on $NE(X)$, where $\sigma\in Aut(X)$ is the automorphism of the cover $\tau$.
Then $X\cong \PP^1\times \PP^1$.
\end{lemma}

\begin{proof}
The Mori cone of a Fano manifold is generated by its extremal rays
and $\sigma_*$ acts by permutation on them. If the permutation is trivial, then $\sigma_*$ is the identity by Lemma \ref{lem:sigmastar}(1).

Thus we can assume that there exists an extremal ray $\Gamma_1$ such that
$\Gamma_2 := \sigma_* \Gamma_1$ is distinct from $\Gamma_1$.
Let $f_i\colon X\to Z_i$ for $i=1,2$ be the two $Aut^{\circ}(X)$-equivariant fibrations induced by the two extremal rays of $NE(X)$.
By Lemma \ref{lemma-embedding} the map
$$
\holom{f:=f_1 \times f_2}{X}{Z_1 \times Z_2}
$$
is an isomorphism onto its image.

By the rigidity lemma there is an isomorphism $\sigma_Z$ making the following diagram commute
$$
\xymatrix{
X\ar[r]^{\sigma}\ar[d]_{f_1}&X\ar[d]^{f_2}\\
Z_1\ar[r]_{\sigma_Z}&Z_2
}
$$
Consider now the involution
$$
\begin{array}{rcl}
i: Z_1\times Z_2&\to &Z_1\times Z_2\\
(z_1,z_2)&\mapsto&(\sigma_Z^{-1}z_2,\sigma_Z z_1)
\end{array}
$$
Let $f$ be the embedding of Lemma \ref{lemma-embedding}.
The image $f(M) \simeq M \subset Z_1 \times Z_2$ is contained in the fixed locus of $i$.

Moreover $i$ is conjugated by the isomorphism $(id,\sigma_Z^{-1})\colon Z_1\times Z_2\to Z_1\times Z_1$ to the involution
$$
\begin{array}{rrcl}
\iota: & Z_1\times Z_1&\to &Z_1\times Z_1\\
& (z_1,z_2)&\mapsto&(z_2,z_1)
\end{array}
$$
The fixed locus of $\iota$ is the diagonal $\Delta \simeq Z_1$ so we obtain that $\dim M \leq \dim Z_1 \leq \dim X-1$. Since $M$ is a divisor in $X$ we have equality
and $M \simeq Z_1$. In particular we have $\rho(M)=\rho(Z_1)=\rho(X)-1$.

By the Lefschetz hyperplane theorem the restriction map $N^1(X) \rightarrow N^1(M)$
is injective if $\dim X \geq 3$. Thus we obtain $\dim X \leq 2$ and therefore $X\cong \PP^1\times\PP^1$ (note that the case of a projective space is excluded by the condition that $\sigma_*$ is not the identity).
\end{proof}

\begin{lemma}\label{lem:Pic3}
Let $X$ be a rational homogeneous space of dimension $n \geq 4$, and
let $M \subset X$ be a smooth ample divisor.
Assume that there is a finite morphism of degree two $\tau \colon X\to Y$ ramified along $M$ such that $\sigma_*$ is the identity on $NE(X)$, where $\sigma\in Aut(X)$ is the automorphism of the cover $\tau$.

Then $-K_X-M$ is ample.
\end{lemma}

\begin{proof}
Since $n \geq 4$ we know by Remark \ref{remark-picard} that $\rho(X)=\rho(Y)$.

Since $M$ is rational homogeneous it is clear that $-K_M=(-K_X-M)|_M$ is ample, but it is not so clear for $-K_X-M$ itself. Arguing by contradiction we set $0 < \lambda \leq 1$
such that $-K_X - \lambda M$ is a nef, but non ample divisor class.
Note that in general $\lambda \in \R$, but since the Mori cone of the Fano manifold
$X$ is rational polyhedral we have $\lambda \in \Q$.
Moreover by the cone theorem there exists an extremal ray $\Gamma$ such that
$(-K_X-\lambda M) \cdot \Gamma=0$, and we denote by
$$
f: X \rightarrow T
$$
its contraction. Since $X$ is rational homogeneous, the fibration is of fibre type and a locally trivial fibration with general fibre $F$ with $\rho(F)=1$ (by Lemma \ref{lem:monodromy}).

Since $\sigma_*$ acts as the identity on $N_1(X)$, the image of the extremal ray
$\tau_* \Gamma$ is an extremal ray in $NE(Y)$ by Lemma \ref{lem:sigmastar}(2), and we denote its
contraction by
$$
g: Y \rightarrow U.
$$
By the rigidity lemma we have an induced morphism $\holom{\bar \tau}{T}{U}$
such that 
$$
\xymatrix{
X \ar[r]^{\tau}\ar[d]_{f}& Y \ar[d]^{g} \\
T \ar[r]_{\bar \tau} & U
}
$$
commutes.
Since $\rho(X)=\rho(Y)$ we have
$$
\rho(T)=\rho(X)-1=\rho(Y)-1=\rho(U).
$$
Thus the morphism $\bar \tau$ is finite, and we claim that is actually an isomorphism.
We argue by contradiction.
Since $\tau$ is two-to-one, it is clear that $\bar \tau$ is then also two-to-one.
In what follows we denote by $G=\fibre{g}{u}$ a general $g$-fibre.

{\em 1st case. Assume that the fibres of $f$ and $g$ have dimension one.}
In this case $G \simeq \PP^1$. Since $\fibre{\bar \tau}{u}$
consists of two points, the preimage $\fibre{\tau}{G}$ has two connected components $F_1 \cup F_2$ mapping isomorphically onto $G$. Yet $-K_X \cdot G=2=-K_Y \cdot F_i$ since they are general fibres of the fibrations. Moreover by the ramification formula
$-K_X = \tau^*(-K_Y) - M$. Since $\tau|_{F_i}$ is an isomorphism we obtain
$M \cdot F_i=0$, a contradiction to the ampleness of $M$.

{\em 2nd case. Assume the fibres of $f$ and $g$ have dimension at least two.}
In this case the restriction of $g$ to $M \subset Y$ has connected fibres, since
the intersection $M \cap G \subset G$ is an ample divisor.
Since $M$ is rational homogeneous (something that we don't know about $Y$), we obtain
that $g|_M: M \rightarrow U$ is smooth and $U$ is smooth and rational homogeneous by the Blanchard's lemma.

The morphism $\bar\tau\colon T\to U$ is a finite morphism of degree two between rational homogeneous, and in particular Fano, varieties. Therefore, it cannot be \'etale.

Since $\bar \tau$ is of degree two and not \'etale there exists by purity of branch a prime divisor
$B \subset U$ such that $\bar \tau^* B$ is not reduced.
Thus $f^* \bar \tau^* B$ has a non-reduced irreducible component.
Since $g$ is an elementary Mori contraction, the preimage $g^{-1} B$ is irreducible.
As a consequence of the Graber-Harris-Starr theorem, the irreducible divisor $g^* B$
is reduced. Yet $\tau^* g^* B = f^* \bar \tau^* B$ is not reduced, so
$g^* B$ is an irreducible component of the branch locus of $\tau$. Yet $\tau$ ramifies
exactly along the ample divisor $M$, a contradiction.

This finishes the proof of the claim that $\bar \tau$ is an isomorphism.

Thus for any point $u \in U$, we have an induced degree two cover
$$
\tau_F: F := \fibre{f}{\bar \tau^{-1}(u)} \rightarrow G := \fibre{g}{u}
$$
that ramifies along $F \cap M$.
Choose $u \in U$ general so that $G$ is smooth.
The fibre $F$ is rational homogeneous by Example \ref{example-fibres}.
Since $\rho(F)=1$ we can apply Corollary \ref{cor:Q2proj}
to obtain that $F$ is a quadric $Q^d$ and $G \simeq \PP^d$. Moreover the divisor
$(F \cap M) \subset F$ is a hyperplane section, i.e. an element of  $|\sO_Q(1)|$.

If $\dim F \geq 2$ this gives an immediate contradiction: by assumption
$K_F+\lambda M_F$ is trivial. Yet for a quadric $-K_Q \simeq \sO_Q(\dim Q)$, so
we obtain $\lambda = \dim Q>1$, a contradiction.

Thus we have $\dim F=1$, in this case the ramification formula yields
$$
2 = -K_X \cdot F = \tau^* (-K_Y) \cdot F - M \cdot F = -K_Y \cdot 2G - M \cdot F = 4 - M \cdot F,
$$
and therefore $M \cdot F=2$. Since $X \rightarrow T$ is a $\PP^1$-bundle and
and $M \rightarrow T$ is finite (since $M$ is rational homogeneous), 
a result of Koll\'ar \cite[Appendix, Lemma (2)]{Bor91} implies $\dim T \leq 2$, the final contradiction.
\end{proof}

\begin{proposition}
Let $X$ be a rational homogeneous space of dimension $n \geq 4$, and
let $M \subset X$ be a smooth ample divisor.
Assume that there is a finite morphism of degree two $\tau \colon X\to Y$ ramified along $M$ such that $\sigma_*$ is the identity on $NE(X)$, where $\sigma\in Aut(X)$ is the automorphism of the cover $\tau$.
Then $\rho(X)=1$.
\end{proposition}
\begin{proof}
Assume by contradiction that $\rho(X)\geq 2$.
Let $\Gamma$ be any extremal ray
in $\NE{X}$, and let $f: X \rightarrow T$ be its contraction. 

{\em 1st case. Assume first that $\dim X-\dim T=1$.}
By Lemma \ref{lem:Pic3} we have $(-K_X-M) \cdot F>0$, which implies that
$$
2 = -K_X \cdot F > M \cdot F \geq 1.
$$
Thus $M \cdot F=1$, and $M$ being rational homogeneous by example \ref{example-degree-two} and by \cite{FH24}, it is a section of $f$.
Yet $M$ is ample, so Koll\'ar's result \cite[Appendix, Lemma (1)]{Bor91} gives $\dim T \leq 1$, a contradiction.

\medskip
{\em 2nd case. Assume now that $\dim X-\dim T>1$.}
Applying Proposition \ref{proposition-contractions}  we see that we even have
$\dim X \geq 2 \dim T$. Since $\rho(X) \geq 2$ there are at least two elementary contractions $f_1$ and $f_2$ and their general fibres $F_1$ and $F_2$ intersect in at most finitely many points. By the first case and the estimate above implies that
$$
\dim F_1 = \dim F_2 = \frac{1}{2} \dim X.
$$

Thus we have $f_i\colon X\to T_i$ such that two general fibres intersect in finitely many points. Then the variety $T_1\times T_2$ is rational homogeneous and $(f_1,f_2)\colon X\to T_1\times T_2$ is equivariant and generically finite. It is surjective by the hypothesis on the dimensions. Therefore $(f_1,f_2)$ is an isomorphism.

\medskip

Since $\sigma_*$ is the identity, for $i=1,2$ there are involutions $\sigma_i$ for $i=1,2$ and a commutative diagram
$$
\xymatrix{
T_1\times T_2\ar[r]^{\sigma}\ar[d]^{p_i}&T_1\times T_2\ar[d]_{p_i}\\
T_i\ar[r]_{\sigma_i}& T_i
}
$$
But this implies that the fixed locus of $\sigma$ is contained in $p_1^{-1}Fix(\sigma_1)\cap p_2^{-1}Fix(\sigma_2) $,
a contradiction because the fixed locus of $\sigma$ is a non-empty divisor.
\end{proof}

We are left with the three-dimensional case.

\begin{proposition}
Let $X$ be a rational homogeneous space of dimension three, and
let $M \subset X$ be a smooth ample divisor that is normal split.
Then one has $\rho(X)=1$.
\end{proposition}

The following proof could be carried out without the classification
of Mori and Mukai \cite{MM81}, but a degree computation in the spirit 
of \cite[Prop.3]{BeaMer87} gives a more streamlined argument. 

\begin{proof}
Assume by contradiction that $\rho(X)\geq 2$ and let $f_i\colon X\to T_i$ for $i=1,2$ be two extremal contractions.
Since the fibres of $f_1$ and $f_2$ intersect in finitely many points there are, up to renumbering 
$f_1$ and $f_2$,  two possibilities : either
$\dim T_1=1$ and $\dim T_2=2$, or $\dim T_1=\dim T_2=2$.

{\em 1st case: assume that $\dim T_1=1$ and $\dim T_2=2$.}
Since $X$ has an elementary Mori contraction onto a curve we have $\rho(X)=2$.
Thus Lemma \ref{lemma-embedding} implies that $X \simeq \PP^1 \times \PP^2$.
Since the fibres of $f_1$ and $f_2$ have different dimensions,
the involution $\sigma_*$ acts as the identity and by Lemma \ref{lem:sigmastar}
we have $\rho(Y)=\rho(X)=2$. 

The ample divisor $M$ is an element of  $|\sO_{\PP^1 \times \PP^2}(a,b)|$ for some $a \geq 1, b \geq 1$.
Since $M$ is rational homogeneous the induced fibration
$$
f_1: M \rightarrow \PP^1
$$
must be a $\PP^1$-bundle structure. In particular by the adjunction formula we have $b \leq 2$. Since $K_M^2=8$ for every Hirzebruch surface a short intersection computation allows to exclude the case $b=2$. Thus we have $b=1$ and it is not difficult to see that
a general element in $\sO_X(a,1)$ with $a \in 2 \N$ is indeed isomorphic to $\PP^1 \times \PP^1$.  We claim that such a $M \subset X$ is never normal split: arguing by contradiction, we apply Proposition \ref{pro:BeaMer} to obtain an 
involution $\sigma: X\rightarrow X$ with fixed locus $M$, and therefore a degree two cover
$$
\holom{\tau}{\PP^1 \times \PP^2}{Y}
$$
that ramifies along $M$. By the ramification formula
$$
\tau^* (-K_Y) \simeq -K_X + M
$$
is ample, so $Y$ is Fano. Moreover we compute
$$
(-K_Y)^3 = \frac{1}{2} (-K_X+M)^3 = 8 (2+a).
$$
Since $(-K_Y)^3 \leq 64$ for any smooth Fano threefold \cite[Cor.11]{MM81}, we see that $a=2$ or $a=4$.

If $a=2$ we have $(-K_Y)^3=32$, so by \cite[Table 2]{MM81} the threefold $Y$
is a blow-up of $\PP^3$. Yet we know that the Mori contractions of $Y$ are fibre spaces, a contradiction.

If $a=4$ we have $(-K_Y)^3=48$, so by \cite[Table 2]{MM81} the threefold $Y$
is the flag manifold $\PP(T_{\PP^2})$. Yet for this threefold both Mori contractions have
$1$-dimensional fibres, a contradiction to $\dim T_1=1$. 

{\em 2nd case: assume that $\dim T_1=\dim T_2=2$.}
The surfaces $T_1$ and $T_2$ are rational homogeneous of the same Picard rank.
Since $M$ is a rational homogeneous surface, we have $\rho(M)\leq 2$.
We get by the Lefschetz theorem
$$
\rho(T_i)=\rho(X)-1 \leq \rho(M)-1\leq 1.
$$
Thus $T_1\cong  T_2\cong \PP^2$.
Now the Mori-Mukai classification (or Lemma \ref{lemma-embedding})
implies that $X \simeq \PP(T_{\PP^2}) \subset \PP^2 \times \PP^2$.

Since $M \subset X$ is ample we have $M \in |\sO_X(a,b)|$ with $a \geq 1, b \geq 1$.
As in Case 1 we use Proposition \ref{pro:BeaMer} to construct a degree two morphism
$\tau: X \rightarrow Y$ ramified along $M$. 
The threefold $Y$ is Fano and
$$
(-K_Y)^3 = \frac{1}{2} (-K_X+M)^3 = \frac{3}{2} (2+a)(2+b)(4+a+b).
$$
Since $a \geq 1, b \geq 1$ we get $(-K_Y)^3 \geq 81$, a contradiction.
\end{proof}

\section{Proof of Theorem \ref{thm:NSdivisor}}

\begin{proof}
Our goal is to reduce Theorem \ref{thm:NSdivisor} to Theorem \ref{thm:NSampledivisor}
using a diagram chase in the spirit of \cite[Section 1.4]{Jah05}.

Since $M$ is effective and $X$ is homogeneous, the divisor $M$ is base-point-free, and induces a morphism
$\varphi\colon X\to U$. Moreover, there is a smooth ample divisor  $N$  on $U$ such that $M=\varphi^{-1}N$. After taking the Stein factorization $\varphi=\nu\circ\psi$ and replacing $\varphi$ with $\psi$
and $N$ with $\nu^{-1}N$, we may assume that $\varphi$ is a fibration.

{\em Step 1. We claim that $N$ is normal split in $U$.}
Consider the diagram
$$
\xymatrix{
T_{M/N}\ar[r]^h\ar[d]_{\alpha_M}&T_{X/U}\ar[d]^{\alpha}&\\
T_{M}\ar[r]_{\tau_M}\ar[d]_{d\varphi_M}&T_{X}\otimes\mathcal O_M\ar[d]^{d\varphi}\ar\ar@/_0.5pc/[l]_s\ar[r]&N_{M/X}\ar[d]\\
\varphi^*T_{N}\ar[r]&\varphi^*T_{U}\otimes\mathcal O_M\ar[r]&\varphi^*N_{N/U}\\
}
$$
where $s$ gives the splitting of the normal sequence of $M$ in $X$.
We notice that $h$ is an isomorphism.
For a local section $v$ of $\varphi^*T_{U}\otimes\mathcal O_M$ we want to define by diagram chase $\tilde s_U(v)=d\varphi(s(w))$
where $w$ is a local section of $T_X$ such that $d\varphi(w)=v$.
To show that $\tilde s_U$ is well defined, we assume $w\in \ker(d\varphi)$. Let $u$ be a local section of $T_{X/U}$
such that $\alpha(u)=w$ and $u'$ such that $h(u')=u$.
We want to prove that
$d\varphi_M(s(w))=0$. We have
$$d\varphi_M(s(w))=d\varphi_M(s(\alpha(h(u'))))=d\varphi_M(s(i(\alpha_M(u'))))=d\varphi_M(\alpha_M(u'))=0.$$
We defined $\tilde s_U\in \Hom(\varphi^*T_{N},\varphi^*T_{U}\otimes\mathcal O_M)\cong H^0(U, \varphi^*(T_{U}\otimes\mathcal O_N\otimes T_{N}^{\vee}))$ such that $\varphi^*\tau_N\circ\tilde s_U=id$.
Since $\varphi$ is a fibration and by the projection formula, this yelds a section $s_U\in H^0(U, T_{U}\otimes\mathcal O_N\otimes T_{N}^{\vee})$  such that $\tau_N\circ s_U=id$.

{\em Step 2. Conclusion.}
We make a case distinction in terms of $n := \dim U$.
\begin{itemize}
\item $n=1$;  equivalently $U \simeq \PP^1$. Then $M$ is a fibre of the locally trivial fibration~$\varphi$ and obviously normal split.
\item $n=2$;  equivalently $U \simeq \PP^2$ or $U \simeq \PP^1 \times \PP^1$.
Then $N$ is completely determined by Van de Ven's theorem and Proposition \ref{proposition-contractions}.
\item Finally if $n \geq 3$ we know by Theorem \ref{thm:NSampledivisor} that $U$ is the projective space or a quadric, and $N$ is a hyperplane section.
\end{itemize}
\end{proof}

\begin{remark} \label{remark-curve-case}
In this paper we have focused on the case of divisors, i.e. the smallest possible codimension. One might ask if the classification can also be carried out for
normal split submanifolds $M \subset X$ of dimension one. In this case $M \simeq \PP^1$, so the only question is to determine the embedding. Unfortunately there are often many possible choices for the embedding: let $X = Q^n \subset \PP^{n+1}$
be a smooth quadric and let $S \subset X$ be a linear section of codimension $n-2$, i.e.
$S \simeq \PP^1 \times \PP^1$ is a smooth quadric surface in $X$. Then $S$ is normal split
in $X$ \cite[Thm.4.5]{Jah05}. Moreover let $M \subset S$ be an irreducible element of $|\sO_{\PP^1 \times \PP^1}(1,a)|$ for {\em any $a \in \N$}. Then $M \subset S$ is normal split by Example \ref{example-fibration}, and therefore $M \subset X$ is normal split by \cite[Sect.1.2]{Jah05}.

In view of this example we think that a more reasonable approach is to classify normal split submanifolds of minimal codimension.
\end{remark}

\bibliographystyle{amsalpha}
\bibliography{Biblio}

\end{document}